\documentclass{article}
\usepackage{amsmath,mathrsfs,amssymb,amsthm,amscd}
\usepackage{color,epsfig,latexsym,graphicx,eucal,layout,fancyhdr}
\usepackage[]{fontenc}
\usepackage[all]{xy}
\usepackage{hyperref}
\usepackage{graphics}
\usepackage{a4wide}
\usepackage{palatino}
\usepackage{setspace}

\newcounter{EQNR}
\renewcommand{\contentsname}{Contents\\
{\footnotesize\normalfont(A table of contents should normally not
be included)}}

 \theoremstyle{plain}
 \newtheorem{thm}{Theorem}[section]
 
 \numberwithin{equation}{section} 
 \theoremstyle{plain}
 \theoremstyle{plain}
 \theoremstyle{definition}
 \newtheorem{defn}[thm]{Definition}
 
 \theoremstyle{plain}
 \newtheorem{prop}[thm]{Proposition}
 \newtheorem{rem}[thm]{Remark}
 \newtheorem{lem}[thm]{Lemma}
 \newtheorem{cor}[thm]{Corollary}
 \newtheorem*{cor*}{Corollary}
 \newtheorem*{conj*}{Conjecture}
 \newtheorem*{thm*}{Theorem}

\newcommand{\bl}{\begin{lem}}
\newcommand{\el}{\end{lem}}

\newcommand{\bml}{\begin{multline}}
\newcommand{\eml}{\end{multline}}

\newcommand{\beq}{\begin{equation}}
\newcommand{\eeq}{\end{equation}}
\newcommand{\bp}{\begin{prop}}
\newcommand{\ep}{\end{prop}}
\newcommand{\bd}{\begin{defn}}
\newcommand{\ed}{\end{defn}}
\newcommand{\pf}{\begin{proof}}
\newcommand{\epf}{\end{proof}}

\newcommand{\field}[1]{\ensuremath{\mathbb{#1}}}

\newcommand{\CC}{\field{C}}

\newcommand{\NN}{\field{N}}

\newcommand{\HH}{\field{H}}

\newcommand{\RR}{\field{R}}

\newcommand{\ZZ}{\field{Z}}

\let\Im\relax
\DeclareMathOperator{\Im}{Im}
\let\Re\relax
\DeclareMathOperator{\Re}{Re}

\DeclareMathOperator{\vol}{vol}

\DeclareMathOperator{\GL}{GL}

\DeclareMathOperator{\tr}{tr}

\newcommand{\hil}{\mathcal{H}}

\DeclareMathOperator{\lp}{\Delta}

\newcommand{\csp}{\textbf{c}}
\newcommand{\kv}{\mathbf{k}}
\newcommand{\elp}{\textbf{e}}

\newcommand{\mm}{\mathfrak{a}}
\newcommand{\Z}{\mathcal{Z}}

\begin{document}

\title{Super-zeta functions and regularized determinants associated to cofinite Fuchsian groups with finite-dimensional unitary representations}
\author{Joshua S. Friedman\footnote{The views expressed in this article are the author's own and not those of the U.S. Merchant Marine Academy,
the Maritime Administration, the Department of Transportation, or the United States government.}, Jay Jorgenson\footnote{Research supported
by several PSC-CUNY grants.} and Lejla Smajlovi\'{c}}

\maketitle

\begin{abstract}\noindent
Let $M$ be a finite volume, non-compact hyperbolic Riemann surface, possibly with elliptic fixed points, and let $\chi$
denote a finite dimensional unitary representation of the fundamental group of $M$.  Let $\Delta$ denote the hyperbolic
Laplacian which acts on smooth sections of the flat bundle over $M$ associated to $\chi$.  From the spectral theory
of $\Delta$, there are three distinct sequences of numbers:  The first coming from the eigenvalues
of $L^{2}$ eigenfunctions, the second coming from resonances associated to the continuous spectrum, and the third being the set of negative integers.  Using these
sequences of spectral data, we employ the super-zeta approach to regularization and introduce two super-zeta functions, $\Z_-(s,z)$ and $\Z_+(s,z)$ that encode the spectrum of $\Delta$ in such a way that they can be used to define the regularized determinant of $\Delta-z(1-z)I$. The resulting formula for the regularized determinant of $\Delta-z(1-z)I$ in terms of the Selberg zeta function, see Theorem 5.3, encodes the symmetry $z\leftrightarrow 1-z$, which could not be seen in previous works, due to a different definition of the regularized determinant.   \end{abstract}

\section{Introduction}

In this article we will develop the super-zeta regularization approach to the spectral data associated to Laplacians which act
on smooth sections of flat vector bundles on a finite volume Riemann surface $M$.  In brief, we will use the super-zeta methodology to define
two functions, $\Z_-(s,z)$ and $\Z_+(s,z)$, carrying the information about the spectrum of the Laplacian in such a way that it makes natural to define regularized determinant of $\Delta - z(1-z)I$ by formula \eqref{detLaplSq} below. We obtain expressions for the Selberg zeta function and the scattering determinant associated to $M$ in terms of $\Z_-(s,z)$ and $\Z_+(s,z)$.  In some sense, our main results complete by different means, the problem studied in \cite{Ef88-91} where the author employed a trace formula approach
to zeta regularization.

If a given sequence $\Lambda =\{\lambda_{j}\}$ is such that $\lambda_{j}$ approaches one sufficiently fast, then it is
an elementary exercise to define the product of elements of $\Lambda$; no regularization is needed.  However, regularization provides a means
by which one can define a product and study its properties when $\Lambda$ is unbounded.
In its most naive interpretation, the regularized product of a sequence, when defined, allows one to write mathematical
expressions of the form ``$\infty ! = \sqrt{2\pi}$'', which is often viewed as amusing when first encountered.  However, upon further reflection,
one views the regularized product of a sequence $\Lambda = \{\lambda_{j}\}$ of numbers to be a special value of the
zeta function $\zeta_{\Lambda}(s) = \sum\lambda_{j}^{-s}$, namely the value $\exp(-\zeta_{\Lambda}'(0))$.  In this form, one
can view the definition of a regularized product as yielding an area of investigation which includes complex analysis, meaning
the study of the function $\zeta_{\Lambda}(s)$ of a complex variable $s \in \mathbb{C}$, as well as a type of analytic number theory,
as it pertains to special values of meromorphic functions.  From this point of view, there exists in the literature
various references which develop the elementary study of regularized products and go so far as to include some interesting
examples; see, for example, \cite{JL93b}.

A somewhat commonplace example for a sequence $\Lambda$ is the set $\{z - \lambda_{j}\}$ where $z$ is a complex variable
and $\{\lambda_{j}\}$ is the sequence of eigenvalues of a self-adjoint operator, such as the Laplacian which acts on smooth
functions on a compact hyperbolic Riemann surface.  In this particular setting, it was shown in \cite{Sarnak} that the
zeta regularized product of $\Lambda$ is closely related to the Selberg zeta function.  Further examples are discussed in
the well-cited article \cite{Hawking77}.  As Hawking discusses on page 141 of \cite{Hawking77}, it is common to develop
a means of regularization by starting with the trace of a heat kernel.  Unfortunately, there are many instances when
such heat kernels are not of trace class, such as when the hyperbolic Riemann surfaces has finite volume yet is not compact.
In recognition of this problem the author of \cite{El12} writes the following on page 7:

\vskip .05in
\it
Notice that the generalization to the case of a continuous spectrum is quite simple (the multi-series being just substituted by a multiple
integral). \rm

\vskip .05in
Respectfully, we disagree with this assertion that zeta regularization is simple in the presence of continuous
spectrum.  Indeed, in \cite{Ef88-91} the author studied in the problem of expressing the Selberg zeta function
as a regularized product, following the approach of \cite{Sarnak}, which employed the Selberg trace formula,
and the concluding result from \cite{Ef88-91} was not entirely successful.  From different studies, various authors
have succeeded in defining \it regularized traces \rm of heat kernels in the setting on finite volume hyperbolic
Riemann surfaces, see for example \cite{JLund97b}, \cite{Mull98} and \cite{MulMul06}, and from these results one is
able to proceed further in developing zeta regularized products.  However, in doing so, one does not see
very clearly the underlying sequence $\Lambda$ since all analytic consideration has been swept into the study of
regularizing the traces of operators.

In the present article we revisit the problem of defining and studying regularized product of the Laplace operator
which acts on the space of smooth sections of flat vector bundles which lie over a finite volume hyperbolic Riemann
surface. Specifically, let $\Gamma$ be a Fuchsian group of the first kind with $\csp$ cusps, and assume $\csp > 0$.  Let $M=\Gamma \backslash \HH$ be the
finite volume, non-compact orbifold quotient space.  Let $\chi$ be finite-dimensional unitary representation of $\Gamma$.  The Laplacian $\Delta$ on $M$,
besides the discrete spectrum, possesses the continuous spectrum with finite multiplicity, which is described through the resonances, meaning the poles of the scattering
matrix.  We recall that the Phillips-Sarnak philosophy \cite{PS92} asserts that for a generic surface there might be no non-trivial $L^{2}$ eigenfunctions; therefore,
any general definition of a determinant of the Laplacian must not solely use eigenvalues of $L^{2}$ eigenfunctions, as is developed in
\cite{Ko91p1} for the special case when $\Gamma$ is arithmetic.

The main purpose of this paper is to undertake a different point of view for regularization which is based on the super-zeta regularization
approach as developed by Voros in \cite{VorosKnjiga}.  Specifically, we define the
square of the (super-zeta regularized) determinant of $\Delta - z(1-z)I$ in such a way that it includes both discrete eigenvalues and resonances.
Furthermore, our expression for $\mathrm{det}^2(\Delta-z(1-z)I)$ encodes the symmetry $z\leftrightarrow 1-z$, which seems necessary based on the
notation but in fact is not true when using a regularized heat trace or Selberg trace formula approach; see \cite{Ef88-91}.

We define two completed zeta functions $Z_+$ and $Z_-$.  The set of zeros $N(Z_+)$  of $Z_+$ contains exactly the non-trivial zeros of the
Selberg zeta function on $M$ stemming from the discrete eigenvalues of the Laplacian and the resonances, while the set of zeros of $Z_-$ is $1-N(Z_+)$.
We then define two super-zeta functions, $\Z_+(s,z)$ and $\Z_-(s,z)$ associated to $Z_+$ and $Z_-$, show that
$\Z_+(s,z)$ and $\Z_-(s,z)$  each  possess a meromorphic continuation to the whole $s$-plane, taking into account
only certain admissible values of $z$.  In addition, the continuations of $\Z_+(s,z)$ and $\Z_-(s,z)$ each are regular at $s=0$, so then
we can define
\begin{equation} \label{detLaplSq}
\mathrm{det}^2(\Delta-z(1-z)I)= \exp\left( {-\frac{d}{ds}\left.\left( \Z_{+}\left( s,z\right)+\Z_{-}\left( s,z\right) \right) \right|_{s=0}  } \right).
\end{equation}
The above definition is further explained in \S~\ref{sec: main res}.   With this, we prove that
$$
\mathrm{det}^2(\Delta-z(1-z)I)= \exp\left(B z + C \right)  \phi(z)\cdot \left( \frac{Z(z)}{G_1(z)(\Gamma(z-1/2))^{\kv}} \right)^2,
$$
where $Z(z)$ is the Selberg zeta function associated to $M$ and $\chi$; $\phi(z)$ is the scattering determinant, $B$ and $C$ are certain explicitly computable constants, and  $G_1(z)$ is a function given in terms of the
Barnes double gamma function and the gamma functions. See \S\ref{sec: main res} and  Definition~\ref{defGamE}  for further details.

As a byproduct of our investigation we deduce the expression of the scattering determinant as a regularized determinant, see Corollary \ref{thm:cont} below.
Note also that the above expression yields the symmetric functional equation for the renormalized Selberg zeta function, see also \cite{Fisher87}.

One of our secondary goals is to allow for ramification points in our Riemann surfaces. Recently, Lee-Peng Teo applied the regularized determinant, in this case, to study the Ruelle zeta-function \cite{Teo20}.

The paper is organized as follows.  In Section 2 we define all notation and state necessary background material from the literature.
It was our aim to make the article self-contained yet not overly lengthy.  In Section 3 we defined and studied the completed
zeta functions which formed the basis of our study.  Since we chose to study finite volume hyperbolic Riemann surfaces with elliptic
points, it was necessary to undertake additional considerations which occur, and those computations constitute a considerable
portion of section 3.  In section 4 we obtain the meromorphic continuation of the super-zeta functions $\Z_+(s,z)$ and $\Z_-(s,z)$
from which we prove the main results of the article, which are stated in Section 5.

\noindent

\section{Background Information}

\subsection{Basic notation} \label{secNotation}
Let $\Gamma\subseteq\mathrm{PSL}(2, \mathbb{R})$ be a Fuchsian group of the first kind acting by fractional
linear transformations on the upper
half-plane $\mathbb{H}:=\{z\in\mathbb{C}\,|\,z=x+iy\,,\,y>0\}$. Let $M$ be the quotient
space $\Gamma\backslash\mathbb{H}$ and $g$ the genus of $M$. Denote by $\csp$ the number of inequivalent cusps
of $M$ and by $\{R\}_{\Gamma}$ the set of inequivalent elliptic classes of elements of $\Gamma$. For a
fixed elliptic representative $R$, we denote by $d_R$ the order of element $R$ and by $\elp$ the cardinality of
the finite set $\{R\}_{\Gamma}$  of inequivalent elliptic classes in $\Gamma$.

Recall that the hyperbolic volume $\mathrm{vol}(M)$ of $M$ is given by the Gauss-Bonnet formula
\begin{align}\label{GaussBon}
\mathrm{vol}(M)=2\pi\bigg(2g-2+\csp + \sum_{\{R\}_{\Gamma}} \left( 1- \frac{1}{d_R}\right)\bigg)=2\pi\bigg( (2g-2 + \csp) + \elp - \sum_{\{R\}_{\Gamma}} \left(\frac{1}{d_R}\right)\bigg).
\end{align}

Let $V$ be an $h-$dimensional complex inner-product space and let $\chi:\Gamma \mapsto \GL(V)$ be a finite-dimensional unitary representation of $\Gamma.$

 Let $\{S_1,S_2,\dots,S_\csp \}$ be parabolic representatives for the cusps of $\Gamma.$  For each $j = 1\dots \csp,$ set
 $$V_j = \{ v \in V ~|~ \chi(S_j) v = v \}.$$
 Let $k_j = \dim(V_j),$ and define the \emph{degree of singularity}, $\kv = \sum_j k_j$. If $\kv = 0,$ we say that $\chi$ is \emph{regular}, see \cite[Section~1.5 p.28]{Fisher87}.

For each  $S_j \in \{S_1,S_2,\dots,S_\csp \},$  let $\lambda_{j1},\ldots, \lambda_{jh}$ be eigenvalues of $\chi(S_j)$ counted with multiplicity.We can write
$$\lambda_{jp} = e^{2\pi i \beta_{jp}}, $$
where $\beta_{jp} = 0$ for $1 \leq p \leq k_j,$ and $\beta_{jp} \in (0,1) $ for $k_j < p \leq h,$ see \cite[Section~1.5 p.30]{Fisher87}.

For $j \in \{1,\dots,\csp \}$ let $$\beta_j = \sum_{p = k_j + 1}^h \beta_{jp}.  $$
Finally we define
\beq a(\chi) = \left( 2^{h\csp}\prod_{j = 1}^\csp \prod_{p=k_j+1}^h \left(\sin(\pi \beta_{jp} \right)\right)^{-1}  \label{defAchi}
\eeq
to be the expression associated to the character $\chi$ and which appears in the right hand side of the second equation from the top of page 71 \cite[Section~2.4 p.71]{Fisher87}. Note that in our paper, $k = 0$ since we do not consider higher weight forms.

Let $\hil(\Gamma,\chi)$ be the associated Hilbert space of square-integrable automorphic functions, and let $\Delta$ be the (non-negative) self-adjoint extension of the Laplacian (see \cite[p. 15-16]{Venkov83}).

\vskip .10in
Given a meromorphic function $f(s)$, we define the \emph{null set} $N(f)$ to be $N(f) = \{ s \in \mathbb{C}~|~f(s)=0\}$ counted with multiplicity. Similarly, $P(f)$
denotes the \emph{polar set,} the set of points where $f$ has a pole.

Our notation is from the well-known sources \cite{Hejhal83}, \cite{Iwa02} and \cite{Venkov83}.

\subsection{The Gamma function} \label{secGamma}
Let $\Gamma(s)$ denote the Gamma function. Its poles are all simple and located at each point of $-\NN,$ where $-\NN = \{ 0,-1,-2,\dots \}$.
For $|\arg{s}| \leq \pi-\delta$ and $\delta > 0$, the asymptotic expansion \cite[p. 20]{AAR99} of $\log{\Gamma(s)}$ is given by
\beq \label{gammaExpan}
\log{\Gamma(s)} = \frac{1}{2}\log{2\pi} + \left(s-\frac{1}{2}\right)\log{s} - s + \sum_{j=1}^{m-1} \frac{B_{2j}}{(2j-1)2j}\frac{1}{s^{2j-1}} + g_{m}(s).
\eeq
Here $B_i$ are the Bernoulli numbers.  Also, for each $m$, $g_{m}(s)$ is a holomorphic function in the right half plane $\Re(s)>0$ such that
$g_{m}^{(j)}(s) = O(s^{-2m+1-j})$ as $\Re(s)\to \infty$ for all integers $j\geq 0$, and where the implied constant depends on $j$ and $m$.

\subsection{The double Gamma function} \label{secBarnes}
The Barnes double Gamma function is an entire order two function defined by
\begin{equation} \label{BarnesDef}
G\left(s+1\right)=\left(2\pi\right)^{s/2}\exp\left[-\frac{1}{2}\left[\left(1+\gamma\right)s^{2}+s\right]\right]\prod_{n=1}^{\infty}\left(1+\dfrac{s}{n}\right)^{n}\exp\left[-s+\frac{s^{2}}{2n}\right],
\end{equation}
where $\gamma$ is the Euler constant. Therefore, $G(s+1)$ has a zero of multiplicity $n,$ at each point $-n \in \{-1,-2,\dots \}.$
For $\Re(s)>0$ and as $s \rightarrow \infty,$ the asymptotic expansion of $\log G(s+1)$ is given in \cite{FL01} or \cite[Lemma 5.1]{AD14} by
\begin{equation} \label{asmBarnes}
\log G(s+1) = \frac{s^2}{2}\left( \log{s} - \frac{3}{2}\right) - \frac{\log{s}}{12} + \frac{s}{2}\log(2\pi)  + \zeta^{\prime}(-1) \>- \\
\sum_{k=1}^{n} \frac{B_{2k+2}}{4\,k\,(k+1)\,s^{2k}} +  h_{n+1}(s).
\end{equation}
Here, $\zeta(s)$ is the Riemann zeta-function and
$$
h_{n+1}(s)= \frac{(-1)^{n+1}}{s^{2n+2}}\int_{0}^{\infty}\frac{t}{\exp(2\pi t) -1} \, \int_{0}^{t^2}\frac{y^{n+1}}{y+s^2} \,dy \,dt.
$$
By a close inspection of the proof of \cite[Lemma 5.1]{AD14}, it follows that $h_{n+1}(s)$ is a holomorphic function in the right half plane $\Re(s)>0$
which satisfies the asymptotic relation $h_{n+1}^{(j)}(s) = O(s^{-2n-2-j})$ as $\Re(s)\to \infty$ for all integers $j\geq 0$,
and where the implied constant depends upon $j$ and $n$.

\subsection{Automorphic scattering determinant} \label{ascatmatrix}
Let $\phi(s)$ denote the determinant of the automorphic scattering matrix $\Phi(s)$ \cite[\S~2.3 and  p. 59]{Venkov83}. Note that in \cite{Venkov83} they denote $\phi(s)$ by $\Delta(s).$

The function $\phi(s)$ is meromorphic of order at most two. Furthermore, $\phi(s)$ is holomorphic for $\Re(s) > \frac{1}{2}$, except for a finite number of poles, and it satisfies the functional equation
\begin{equation}\label{functeq phi}
\phi(s)\phi(1-s)=1.
\end{equation}


\begin{thm}  (\cite[Thm.~3.5 p. 59]{Venkov83}) For $\Re(s)> 1$ we have that
\begin{equation} \label{phiDirich}
\phi (s)=\left( \frac{ \sqrt{\pi} \,\Gamma \left( s-\frac{1}{2}
\right) }{\Gamma \left( s\right) }\right) ^{\kv}\overset{\infty }{\underset
{n=1}{\sum }}\frac{d(n) }{g_{n}^{2s}}
\end{equation}%
where $0< g_{1} < g_{2}< ...$ and $d(n) \in \mathbb{C}$ with $d(1)\neq 0$.\footnote{The scattering determinant $\phi(s)$ is actually real valued on $\mathbb{R}$ and it follows that $d(n) \in \mathbb{R}.$ }
\end{thm}

We will rewrite \eqref{phiDirich} in a slightly different form. Let $c_{1}=-2\log{g_{1}} \neq 0,$  $c_{2}=\log d(1),$  and let $u_{n}=g_{n}/g_{1}>1$.
Then for $\text{Re}(s) > 1$ we can write $\phi( s) =L(s)H(s)$ where
\begin{equation} \label{eqPhiA}
L(s) =\left( \frac{ \sqrt{\pi} \,\Gamma \left( s-\frac{1}{2}
\right) }{\Gamma \left( s\right) }\right) ^{\kv} e^{c_{1}s+c_{2}}
\end{equation}
and
\begin{equation} \label{eqPhiB}
H(s) =1+\overset{\infty }{\underset{n=2}{\sum }}\frac{
a\left( n\right) }{u_{n}^{2s}},
\end{equation}
where  $a(n) \in \mathbb{C}.$ The series \eqref{eqPhiB} converges absolutely for $\Re(s)>1$.
From the generalized Dirichlet series representation \eqref{eqPhiB} of $H(s)$, it follows that
\beq \label{asmPhi}
\frac{d^{k}}{ds^{k}}\log{H(s)} = O(\beta_k^{-\Re(s)}) \quad \textrm{\rm when} \quad \Re(s) \to +\infty,
\eeq
for some $\beta_k > 1$  where the implied constant depends on $k \in \NN.$

For $0 \leq \sigma_i \leq 1,$ define $q(\sigma_i)=\text{[The multiplicity of the pole of $\phi(s)$ at $s = \sigma_{i}$] }.$

The divisor of $\phi(s)$ consists of the following sets of points \cite[pp. 59--60]{Venkov83}:

\begin{enumerate}
\item Finitely many real zeros of the form $1-\sigma_i \in [0,1/2)$ for $i=1\dots T$, each with multiplicity $q(\sigma_i)$;
\item Finitely many real zeros of the form $\rho_i > 1/2$, $i=1\dots N,$ where $N$ is defined to be the sum of the multiplicities;
\item Finitely many real poles of the form $1-\rho_i < 1/2$, where  $i=1\dots N;$
\item Finitely many poles $\sigma_i \in (1/2,1].$ Each pole has multiplicity $q(\sigma_i).$
\item Poles of the form $1-\rho$ and $1-\overline{\rho}$ with $\Re(\rho) > 1/2$ and $\Im(\rho) > 0;$
\item Zeros of the form $\rho$ and $\overline{\rho}$ with $\Re(\rho) > 1/2$ and $\Im(\rho) > 0$.
\end{enumerate}


Let $\lambda_i$ be an eigenvalue for the positive, self-adjoint extension $\Delta$ of the hyperbolic Laplacian.
Denote by $A(\lambda_i)$ the $\Delta-$eigenspace corresponding to the eigenvalue $\lambda_i.$
Set $A_{1}(\lambda_{i})$ to be the subspace of $A(\lambda_i)$ that is spanned by the incomplete theta series.
For each pole $\sigma_{i} \in (1/2,1]$, $i=1,...,T$  the space $A_{1}(\sigma_{i}(1-\sigma_{i}))$   is non-trivial. In fact
from (\cite[Eq. 3.33 on p.299]{Hejhal83}) we have that
\begin{equation*}
q(\sigma_i)=\text{[The multiplicity of the pole of $\phi(s)$ at $s = \sigma_{i}$] } \leq \dim A_{1}(\sigma_{i}(1-\sigma_{i})) \leq \kv.\label{eqBndPole}
\end{equation*}

The eigenvalue $\lambda_i = \sigma_{i}(1-\sigma_{i})$ is called a \emph{residual eigenvalue.}

\subsection{Selberg zeta-function} \label{szeta}
The Selberg zeta function associated to the quotient space $M=\Gamma\backslash\mathbb{H},$ and unitary representation $\chi:\Gamma \mapsto \GL(V)$ is defined for $\Re(s)>1$ by
the absolutely convergent Euler product
\begin{equation*}
Z(s)=\prod\limits_{\left\{ P_0\right\} \in P(\Gamma
)}\prod_{n=0}^{\infty }\det\left( 1_V-\chi(P_0) N(P_0)^{-(s+n)}\right) \text{,}
\end{equation*}
where $P(\Gamma )$ denotes the set of all primitive hyperbolic conjugacy classes in $\Gamma,$ and $N(P_0)$ denotes the norm of $P_0 \in \Gamma.$
From the product representation given above, we have for $\Re(s)>1$ that

\begin{equation*} \label{log z(s)}
\log{Z(s)} = \sum_{\left\{ P_0\right\} \in P(\Gamma
)} \sum_{n=0}^\infty \tr \left(-\sum_{l=1}^\infty \frac{\chi(P_0)^l}{l} N(P_0)^{-(s+n)l}  \right) = -
\sum_{P\in H(\Gamma )}\tr(\chi(P))\frac{\Lambda (P)}{N(P)^{s}\log N(P)},
\end{equation*}
where $H(\Gamma )$ denotes the set of all hyperbolic conjugacy classes in $\Gamma,$ and $\Lambda (P)=\frac{\log N(P_{0})}{1-N(P)^{-1}}$, for the
 primitive element $P_{0}$ in the conjugacy class containing $P$ (see \cite[p. 83]{Venkov83}).

Let $P_{00}$ be the primitive hyperbolic conjugacy class in all of $P(\Gamma )$ with the smallest norm. Setting $\alpha = N(P_{00})^{\tfrac{1}{2}}$, for $\Re(s) > 2$ and $k \in \NN$ we have the following asymptotic formula
\beq \label{eqSelZetaBound}
\frac{d^{k}}{ds^{k}}\log{Z(s)} = O(\alpha^{-\Re(s)}) \quad \textrm{\rm when} \quad \Re(s) \to +\infty,
\eeq
with an implied constant which depends on $k \in \NN.$

If $\lambda_j$ is an eigenvalue in the discrete spectrum of $\Delta$, let  $m(\lambda_j)$ denote its multiplicty.
We now state the divisor of the $Z(s)$  (see \cite[p. 49]{Venkov90}  \cite[p. 499]{Hejhal83}):

\begin{enumerate}
\item Zeros at the points  $s_j$ on the line $\Re(s)=\tfrac{1}{2}$ symmetric relative to the real axis and in  $(1/2,1]$, where each zero $s_j$ has multiplicity $m(s_j) = m(\lambda_j)$ where $s_j(1-s_j) = \lambda_j$ is an eigenvalue in the discrete spectrum of $\Delta$  \label{szeta1};

\item Zeros at the points $s_{j} \in [0,1/2)$ where $s_j(1-s_j) = \lambda_j \in[0,1/4)$ is an eigenvalue in the discrete spectrum of $\Delta$ and
the multiplicity $\widetilde{m}(s_{j})$ is given by $\widetilde{m}(s_{j}) =  m(\lambda_j) - q(1-s_j) \geq 0$; we denote by $K$ the number of eigenvalues $\lambda_j \in[0,1/4)$ and put $m_j=m(\lambda_j)$, $j=1,...K $ ($q(\cdot)$ was defined in \S\ref{ascatmatrix}).

Note that, in the case when $\lambda_j$ is not the residual eigenvalue, we take $q(1-s_j)= 0$, i.e. $\widetilde{m}(s_{j}) =  m(\lambda_j)$.

\item The point $s=\tfrac{1}{2}$ can be a zero or a pole, and the order of the point as a divisor is
$$\mm = 2d_{1/4}- \tfrac{1}{2}\left( \kv- \tr \Phi (\tfrac{1}{2})\right)$$
where $d_{1/4}$ denotes the multiplicity  of the possible eigenvalue $\lambda = \tfrac{1}{4}$
of $\Delta$;
\item Poles at $s=-n-\tfrac{1}{2}, $ where $n=1,2,\dots,$  each with
multiplicity $\kv$;
\item Finitely many real zeros $1-\rho_i < 1/2$, where  $i=1\dots N;$
\item Zeros at each $s = 1-\rho, 1-\overline{\rho}$ where $\rho$ is a zero of $\phi(s)$ with $\Re(\rho) > \tfrac{1}{2}$ and $\Im(\rho)>0;$ \
 \item Zeros at points $s=-n \in -\NN =  \{0,-1,-2,\dots\}$, with multiplicities
    \begin{equation} \label{EllipZero}
   m_n= h\frac{\vol(M)}{2\pi }(2n+1)- \sum_{\{R\}_{\Gamma}}\sum_{k=1}^{d_R-1} \frac{\tr(\chi^{k}(R))}{d_R} \frac{\sin\left(\frac{k\pi (2n+1)}{d_R}\right)}{\sin\left(\frac{k\pi}{d_R}\right)}.
   \end{equation}
\end{enumerate}

The last set of zeros are called \emph{trivial} zeros.  For the purposes of our paper it will be  crucial that we give another representation of $m_n$ in such a way that it is clear that they are non-negative integers. Ultimately, we will construct a double gamma based function whose divisor is exactly $m_n.$ Finally, note that we will see that, actually $m_0=0$.

\section{Construction of the complete zeta functions}

\subsection{The trivial zeros stemming from the identity motion and elliptic elements}
Recall \eqref{EllipZero}. We construct an entire function on $\CC$ with zeros at the points $-n \in -\NN$ with multiplicities $m_n,$ following ideas of Fisher \cite{Fisher87}.

\begin{lem}, \label{lemFinEllSum}
 Suppose that $\omega$ is a $d\mathrm{th}-$root of unity, and define the integer $q$ by
$\omega =  \exp\left(\frac{2 \pi i q}{d}\right),$ with $0 \leq q \leq d-1.$ For $n \in \NN,$ let $q(n) \in\{0,\ldots,d-1\}$ be the residue of $n+q$ modulo $d$ and let $\tilde{q}(n) \in \{0,\ldots,d-1 \}$ be the residue of $n-q$ modulo $d.$ Then we have
$$\sum_{k=1}^{d-1} \omega^k \left( \frac{\sin\left(\frac{k \pi (2n+1)}{d}\right)}{\sin\left(\frac{k\pi}{d}\right)}\right)  =  d-1 - \left(q(n) + \tilde{q}(n)\right)$$
\end{lem}
\begin{proof}
We follow \cite[pages 66--67]{Fisher87}.\footnote{Note that our notation is very different from his. Note that his $k$ represents the weight of his forms, which we take as zero in our paper.}

\begin{multline}
\sum_{k=1}^{d-1} \omega^k \frac{i e^{-i\frac{\pi k}{d} (2n+1) } }{2\sin\left(\frac{k\pi}{d}\right)}= \sum_{k=1}^{d-1} \frac{ e^{-i\frac{2 \pi k}{d} (q + n) } }{1 - e^{\frac{i k 2 \pi}{d} }} = \lim_{t \rightarrow 1^-} \sum_{m=0}^\infty \sum_{r=0}^{d-1} t^{r+md}\sum_{k=1}^{d-1} e^{-i\frac{2 \pi k}{d} (q + n - r - md) } \\ =  \lim_{t \rightarrow 1^-} \frac{1}{1-t^d}\left(t^{q(n)}d - \sum_{r=0}^{d-1}t^r  \right) = \frac{1}{2}(d-1) - q(n), \label{eqTrigSum1}
\end{multline}
where the last equality follows from
$$\sum_{k=1}^{d-1} e^{-i\frac{2 \pi k}{d} (q + n - r - md) } =  \begin{cases}
      d-1, & \text{if} \quad r = n + q \mod d \\
      -1 & \text{else.}
   \end{cases}$$
A similar computation yields
\beq \sum_{k=1}^{d-1} \omega^k \frac{i e^{i\frac{\pi k}{d} (2n+1) } }{2\sin\left(\frac{k\pi}{d}\right)} = -\frac{1}{2}(d-1) + \tilde{q}(n).  \label{eqTrigSum2} \eeq
The lemma now follows by subtracting \eqref{eqTrigSum2} from   \eqref{eqTrigSum1}.
\end{proof}
Assume $R$ is an order $d_R$ elliptic element of $\Gamma$. Recall that $\chi(R)$ is unitary, acting on the $h-$dimensional space $V.$ We will need to simplify expressions of the form
$\tr(\chi^{k}(R)).$ Since $\chi(R)$ is can be diagonalized, it follows that
\beq \label{eqDiagEig}
\tr(\chi^{k}(R)) = \sum_{j=1}^h \omega(R)^k_j,
\eeq
where $\omega(R)_j,$ for $j=1\dots h$ are the eigenvalues (and $d_R\mathrm{th}-$roots of unity) of $\chi(R).$

For each class $\{R\}$ and $j=1\dots h,$ define integer $q(R)_j$, $0 \leq q(R)_j \leq d_R-1$ by
\begin{equation}\label{def:q(R)j}
\omega(R)_j =  \exp\left(\frac{2 \pi i q(R)_j}{d_R}\right).
\end{equation}

\begin{defn}
For an elliptic representative $R$ in ${\{R\}_\Gamma},$ and $m \in \NN,$ define $q_j(R,m), \tilde{q}_j(R,m),$ and $k_j(R,m), \tilde{k}_j(R,m) \in \ZZ$ by
$$q_j(R,m) := m+q(R)_j + d_R k_j(R,m) \in \{0,\ldots,d_R-1\},$$
$$\tilde{q}_j(R,m) := m- q(R)_j + d_R \tilde{k}_j(R,m) \in \{0,\ldots,d_R-1\},$$
define
\beq
\alpha(R,m) := \sum_{j=1}^h (\tilde{q}_j(R,m) +q_j(R,m)), \label{defAlpha}
\eeq
  $$k(R,m,j) := k_j(R,m) + \tilde{k}_j(R,m), $$
  $$\beta(R,m) := \sum_{j=1}^h k(R,m,j). $$
\bl With the notation above,
\begin{equation}\label{alpha of R,m}
\alpha(R,m)=2mh+d_R\sum_{j=1}^{h}k(R,m,j) =2mh +\beta(R,m)d_R,
\end{equation}

\beq \label{defKrmj}
 k(R,m,j)=\left\{
               \begin{array}{ll}
         1, & \mathrm{\, if\,  } m<q(R)_j\, \mathrm{and }\, m+q(R)_j<d_R \\
          -1, & \mathrm{\, if\,  } m\geq q(R)_j\, \mathrm{and }\, m+q(R)_j \geq d_R\\
        0, & \mathrm{\, otherwise  }.
        \end{array}. \right.
\eeq

\el
\pf
Equation~\ref{alpha of R,m} follows immediately. For Equation~\ref{defKrmj}, note that $k_j(R,m)d_R, \tilde{k}_j(R,m)d_R$ are the multiples of $d_R$ that translate $m+q(R)_j, m- q(R)_j$ back to the range $\{0,\ldots,d_R-1\}.$ Noting that $0 \leq q(R)_j \leq d_R-1,$ the derivation of Equation~\ref{defKrmj} is straightforward.
\epf

\begin{rem}
In case when $\chi$ is trivial, $\beta(R,m)=0$ for all $m$ and all expressions above are significantly simplified.
\end{rem}

\end{defn}
We give an equivalent expression for $m_n.$ Applying Lemma~\ref{lemFinEllSum}
$$
\sum_{k=1}^{d_R-1} \tr(\chi^{k}(R))\frac{\sin\left(\frac{k\pi (2n+1)}{d_R}\right)}{\sin\left(\frac{k\pi}{d_R}\right)} = h(d_R-1)-\sum_{j=1}^h (\tilde{q}_j(R,n) +q_j(R,n)),
$$

and also using \eqref{GaussBon}, we rewrite \eqref{EllipZero} as

\begin{multline} \label{eqB392}
m_n = h \bigg( 2g-2 + \csp + \elp - \sum_{\{R\}_{\Gamma}} \left(\frac{1}{d_R}\right)\bigg)(2n+1) - \sum_{\{R\}_\Gamma}\frac{1}{d_R}\left(h(d_R-1)-\sum_{j=1}^h (\tilde{q}_j(R,n) +q_j(R,n)) \right) \\ =  h \bigg( 2g-2 + \csp + \elp \bigg)(2n+1) - \sum_{\{R\}_\Gamma}\frac{1}{d_R}\left(h(2n +d_R)-\sum_{j=1}^h (\tilde{q}_j(R,n) +q_j(R,n)) \right) \\ =
 h \bigg( 2g-2 + \csp + \elp \bigg)(2n+1) - \sum_{j=1}^h\sum_{\{R\}_\Gamma}\frac{1}{d_R}\left(2n +d_R-(\tilde{q}_j(R,n) +q_j(R,n)) \right)
\end{multline}

Consider the function
$$
 f_R(s) := \frac{(\Gamma(s))^{d_R}}{(G(s+1))^2} \prod_{m=0}^{d_R - 1} \Gamma\left(\frac{s+m}{d_R}\right)^{-(\tilde{q}_j(R,m) +q_j(R,m))}
$$

\bl The principal branch of $(f_R(s))^{1/d_R}$ is an meromorphic function on all of $\CC$  with zeros (or poles) of order
$$-\frac{1}{d_R}\left(2n +d_R-(\tilde{q}_j(R,n) +q_j(R,n)) \right)$$ at $s=-n,$ for $n \in \NN.$
\el
\begin{proof}
If follows from the definitions that the order of $f_R(s)$ at $s=-n$ is
$$-\left(2n +d_R-(\tilde{q}_j(R,n) +q_j(R,n)) \right).$$
It follows from the definitions of $\tilde{q}_j(R,n), q_j(R,n)$ that
$$d_R~|~\left(2n - (\tilde{q}_j(R,n) +q_j(R,n))\right).$$
Since $f_R(s)$ is meromorphic on $\CC$ it can be written as a quotient of two entire functions $g_R(s)/h_R(s)$ whose zero sets are disjoint. Hence, the order of each zero of both $g_R,h_R$ must be divisible by $d_R.$ Finally, using the Weierstrass factorization theorem one could construct an entire $d_R$-th root of  $g_R,h_R,$ and $f_R$.
\end{proof}

From the right hand side of \eqref{eqB392} we immediately obtain
\bl The following meromorphic function on $\CC$
$$G_0(s) := \left(\frac{(G(s+1))^2}{(\Gamma(s))}\right)^{h(2g-2 + \csp + \elp )}  \prod_{j=1}^h \prod_{\{R\}_\Gamma}  \left( \frac{(\Gamma(s))^{d_R}}{(G(s+1))^2} \prod_{m=0}^{d_R - 1} \Gamma\left(\frac{s+m}{d_R}\right)^{-(\tilde{q}_j(R,m) +q_j(R,m))}\right)^{\tfrac{1}{d_R}}$$
has zeros or order $m_n$ at $s=-n,$ for $n \in \NN.$
\el

For  representative $R$ in ${\{R\}_\Gamma}$ and $m \in \NN,$ Using the Gauss-Bonnet formula \eqref{GaussBon}, we can rewrite $G_0(s)$ as
$$
G_0(s) =  \left(\frac{ (G(s+1))^2}{\Gamma(s)}\right)^{h \frac{\vol(M)}{2\pi }}\prod_{\{R\}_\Gamma}  \left(\Gamma(s)\right)^{h(1-1/d_R)}\prod_{m=0}^{d_R - 1} \Gamma\left(\frac{s+m}{d_R}\right)^{-\alpha(R,m)/d_R}.
$$
Note that the fractional powers of $G(s+1)$ and $\Gamma(s)$ are defined via the principal branch of $\log(z).$

Finally, set
\begin{equation}\label{defn:g1}
G_1(s)= \left(\frac{(2\pi)^{-s} (G(s+1))^2}{\Gamma(s)}\right)^{h \frac{\vol(M)}{2\pi }}\prod_{\{R\}_\Gamma}  d_R^{-h(1-1/d_R)s}\left(\Gamma(s)\right)^{h(1-1/d_R)}\prod_{m=0}^{d_R - 1} \Gamma\left(\frac{s+m}{d_R}\right)^{-\alpha(R,m)/d_R}.
\end{equation}

Note that $G_1$ is an meromorphic function on $\CC$ of order two with zeros (or poles) at points $-n \in -\NN$ and corresponding multiplicities  $m_n$. Also note that we added an exponential normalization factor that will later simplify some computations.

\subsection{Asymptotic expansion of $G_1$}

In order to derive our main results, we need the asymptotic expansion of function $\log G_1(s)$, as $s\to\infty$ which is given in the following lemma.

\begin{lem}
As $s\to\infty$ we have the following asymptotic expansion of $\log G_1(s)$:
\begin{multline} \label{eqExpG1}
\log{G_1(s)}  =   h\frac{\vol(M)}{2\pi}s^{2}(\log(s)-\frac{3}{2})+\widetilde{a}_1 s\left( \log(s)-1\right) +b_1s +\widetilde{a}_0\log(s)+b_{0} + O(s^{-1}),
\end{multline}
where
\begin{equation}\label{a1tilde def}
\widetilde{a}_1=- h\frac{\vol(M)}{2\pi}-\sum_{\{R\}_\Gamma}\sum_{m=0}^{d_R-1}\frac{\beta(R,m)}{d_R},
\end{equation}
\begin{equation}\label{b1 def}
b_1=\sum_{\{R\}_\Gamma}\sum_{m=0}^{d_R-1}\frac{\beta(R,m)}{d_R}\log d_R
\end{equation}
and contants $\widetilde{a}_0$ and $b_0$ will be explicitly given in the proof, equations \eqref{widetilde a0} and \eqref{b0}, respectively.
\end{lem}

\begin{proof}
Using \eqref{alpha of R,m} we obtain

\begin{multline} \label{starting eq}
\log \left( \prod_{\{R\}_\Gamma} \prod_{m=0}^{d_R - 1} \Gamma\left(\frac{s+m}{d_R}\right)^{-\alpha(R,m)/d_R} \right) = -\sum_{\{R\}_\Gamma}\sum_{m=0}^{d_R-1}\frac{2mh}{d_R}\log\Gamma\left(\frac{s}{d_R}+\frac{m}{d_R}\right)  \\ - \sum_{\{R\}_\Gamma}\sum_{m=0}^{d_R-1}\beta(R,m)\log\Gamma\left(\frac{s}{d_R}+\frac{m}{d_R}\right)
\end{multline}

For any real numbers $d>1$ and $a\geq0$, equation \eqref{gammaExpan} yields the following expansion as $s\to\infty$
\begin{equation*}\label{gammaFractExp}
\log \Gamma\left(\frac{s}{d}+\frac{a}{d}\right) = \frac{1}{d}s(\log s -1)-\frac{\log d}{d}s+\left( \frac{a}{d}-\frac{1}{2}\right) \log s +\frac{1}{2}\log(2\pi d)-\frac{a}{d}\log d + O(s^{-1}).
\end{equation*}

Therefore
\begin{eqnarray}\label{gammaFractExp2}
\sum_{\{R\}_\Gamma}\sum_{m=0}^{d_R-1}\frac{2mh}{d_R}\log\Gamma\left(\frac{s}{d_R}+\frac{m}{d_R}\right)=
\sum_{\{R\}_\Gamma}h\left(1-\frac{1}{d_R}\right)\left[s(\log s-1) - s \log d_R \right] + \\
+\sum_{\{R\}_\Gamma}h\frac{(d_R-1)(d_R-2)}{6d_R}\log s + \sum_{\{R\}_\Gamma}h(d_R-1)\left( \frac{1}{2}\log(2\pi d_R)-\frac{2d_R-1}{3d_R}\log d_R\right) + O(s^{-1})\notag
\end{eqnarray}
and
\begin{multline}\label{gammaFractExp3}
\sum_{\{R\}_\Gamma}\sum_{m=0}^{d_R-1}\beta(R,m)\log\Gamma\left(\frac{s}{d_R}+\frac{m}{d_R}\right)=
\sum_{\{R\}_\Gamma}\sum_{m=0}^{d_R-1}\frac{\beta(R,m)}{d_R}s(\log s-1) - s \sum_{\{R\}_\Gamma}\sum_{m=0}^{d_R-1}\frac{\beta(R,m)}{d_R}\log d_R \\
+\sum_{\{R\}_\Gamma}\sum_{m=0}^{d_R-1}\frac{\beta(R,m)}{d_R}\left(\frac{m}{d_R}-\frac{1}{2}\right)\log s + \sum_{\{R\}_\Gamma}\sum_{m=0}^{d_R-1}\beta(R,m)\left( \frac{1}{2}\log(2\pi d_R)-\frac{m}{d_R}\log d_R\right) + O(s^{-1})
\end{multline}

Next, using \eqref{asmBarnes} and \eqref{gammaExpan} we obtain
\begin{multline}
\log \left(  \left(\frac{(2\pi)^{-s} (G(s+1))^2}{\Gamma(s)}\right)^{h \frac{\vol(M)}{2\pi }} \right) =  h\frac{\vol(M)}{2\pi} \cdot \\ \cdot \left( s^2 \left(\log(s)-\frac{3}{2} \right) - s(\log(s) -1) + \frac{1}{3}\log(s) + 2\zeta'(-1)-\frac{1}{2}\log(2\pi) \right) + O(s^{-1})
\end{multline}
and
\begin{multline} \label{lasteq}
\log \left( \prod_{\{R\}_\Gamma}  d_R^{-h(1-1/d_R)s}\left(\Gamma(s)\right)^{h(1-1/d_R)} \right) \\ = \sum_{\{R\}_\Gamma} \left[  h(1-1/d_R) \left( \frac{1}{2}\log{2\pi} + (s-\frac{1}{2})\log{s} - s \right) -h(1-1/d_R)s\log(d_R) \right] + O(s^{-1}).
\end{multline}

Finally, combining \eqref{starting eq}--\eqref{lasteq} with \eqref{defn:g1} we arrive at the expansion
\begin{multline*}
\log{G_1(s)}  = h\frac{\vol(M)}{2\pi} \left( s^2 \left(\log(s)-\frac{3}{2} \right) - s(\log(s) -1) + \frac{1}{3}\log(s) + 2\zeta'(-1)-\frac{1}{2}\log(2\pi)     \right)  \\
- \sum_{\{R\}_\Gamma}\sum_{m=0}^{d_R-1}\frac{\beta(R,m)}{d_R}s(\log s-1) +s \sum_{\{R\}_\Gamma}\sum_{m=0}^{d_R-1}\frac{\beta(R,m)}{d_R}\log d_R
\end{multline*}
\begin{multline*}+\left( h \sum_{\{R\}_\Gamma}\frac{d_R-1}{d_R}\left(\frac{1}{2}-\frac{d_R-2}{6}\right) - \sum_{\{R\}_\Gamma}\sum_{m=0}^{d_R-1}\frac{\beta(R,m)}{d_R}\left(\frac{m}{d_R}-\frac{1}{2}\right)\right)\log s \\ + h \sum_{\{R\}_\Gamma}(d_R-1)\left(\frac{1}{2d_R}\log(2\pi)-\frac{1}{2}\log(2\pi d_R) + \frac{2d_R-1}{3d_R}\log d_R\right)\\-\sum_{\{R\}_\Gamma}\sum_{m=0}^{d_R-1}\beta(R,m)\left( \frac{1}{2}\log(2\pi d_R)-\frac{m}{d_R}\log d_R\right) + O(s^{-1}),
\end{multline*}
hence
\begin{multline*}
\log{G_1(s)}  =   h\frac{\vol(M)}{2\pi}s^{2}(\log(s)-\frac{3}{2})+\widetilde{a}_1 s\left( \log(s)-1\right)  +b_1 s + \widetilde{a}_0\log(s)+b_{0} + O(s^{-1}),
\end{multline*}
and it is left to give expressions for $\widetilde{a}_0$ and $b_0$:
\begin{equation}\label{widetilde a0}
\widetilde{a}_0= h \sum_{\{R\}_\Gamma}\frac{d_R-1}{d_R}\left(\frac{1}{2}-\frac{d_R-2}{6}\right) - \sum_{\{R\}_\Gamma}\sum_{m=0}^{d_R-1}\frac{\beta(R,m)}{d_R}\left(\frac{m}{d_R}-\frac{1}{2}\right) + h\frac{\vol(M)}{6\pi},
\end{equation}
\begin{multline}\label{b0}
b_0=h \sum_{\{R\}_\Gamma}(d_R-1)\left(-\frac{1}{2d_R}\log(2\pi)-\frac{1}{2}\log(2\pi d_R) + \frac{2d_R-1}{3d_R}\log d_R\right)\\ + h\frac{\vol(M)}{2\pi}\left(2\zeta'(-1)-\frac{1}{2}\log(2\pi) \right)-\sum_{\{R\}_\Gamma}\sum_{m=0}^{d_R-1}\beta(R,m)\left( \frac{1}{2}\log(2\pi d_R)-\frac{m}{d_R}\log d_R\right).
\end{multline}

\end{proof}

\subsection{Complete zeta functions}  \label{complete zetas}

\begin{defn} \label{defCompleteZeta} We define completed zeta functions $Z_+$ and $Z_-$ as
$$
Z_+(s)=\frac{Z(s)}{G_1(s)(\Gamma(s-1/2))^{\kv}}, \quad Z_-(s)= Z_+(s)\phi(s),
$$
where $G_1(s)$ is defined in \eqref{defn:g1} and $\phi(s)$ is the scattering determinant.
\end{defn}
Note that we have canceled out the trivial zeros and poles of the Selberg zeta-function $Z(s)$.
Hence the zero set $N(Z_+)$ of $Z_{+}$consists of the following points:
\begin{enumerate}
\item At $s = \tfrac{1}{2}$ with multiplicity $\mm$ where
$$
\mm = 2d_{1/4}+\kv-\tfrac{1}{2}\left( \kv-\tr \Phi (\tfrac{1}{2})\right) = 2d_{1/4}+\tfrac{1}{2} \left( \kv +\tr \Phi (\tfrac{1}{2})\right) \geq 0;
$$
\item At the points $s_{j} \in [0,1/2)$ where $s_j(1-s_j) = \lambda_j$ is an eigenvalue in the discrete spectrum of $\Delta$
each with multiplicity $m(\lambda_j) - q(1-s_j) \geq 0;$
\item At the points $s_j$ on the line $\Re(s)=\tfrac{1}{2}$ symmetric relative to the real axis and in  $(1/2,1]$  where
each zero $s_j$ has multiplicity $m(s_j) = m(\lambda_j)$ where $s_j(1-s_j) = \lambda_j$ is an eigenvalue in the discrete spectrum of $\Delta$;
\item At each point $s = 1-\rho,1-\overline{\rho}$ where $\rho$ is a zero of $\phi(s)$ with $\Re(\rho) > \tfrac{1}{2},$ and $\Im(\rho) > 0.$
\end{enumerate}


From the definition of $Z_-$, it immediately follows that $N(Z_-) = 1-N(Z_+)$. In other words,
$s$ is a zero of $Z_+$ if and only if $1-s$ is a zero, necessarily with the same multiplicity, of $Z_{-}$.

\section{Superzeta functions associated to complete zeta functions $Z_+$ and $Z_-$}

In this section we define superzeta functions associated to completed zeta functions $Z_+$ and $Z_-$ and show that they possess a meromorphic continuation to the whole complex plane, regular at zero.

\subsection{Regularized products using superzeta functions}
Let $\RR^{-} = (-\infty,0]$ be the non-positive real numbers. Let $\{y_{k}\}_{k\in \mathbb{N}}$ be the sequence of zeros
of an entire function $f$ of order at most two, repeated with their multiplicities. Let
$$X_f = \{z \in \CC~|~ (z-y_{k}) \notin \RR^{-}~\text{for all} ~ y_{k} \}. $$
For $z \in X_f,$ and $s \in \CC$ consider the series
\begin{equation}
\Z_{f}(s,z)=\sum_{k=1}^{\infty }(z-y_{k})^{-s},  \label{Zeta1}
\end{equation}
where the complex exponent is defined using the principal branch of the logarithm with $\arg z\in
\left( -\pi ,\pi \right) $ in the cut plane $\CC \setminus \RR^{-}$.
Since $f$ is of order at most two, the series $\Z_{f}(s,z)$ converges absolutely for $\Re(s) > 2.$
Following \cite{Voros1}, the series $\Z_{f}(s,z)$ is called the \it superzeta function \rm
associated to the zeros of $f $, or the simply the \emph{superzeta} function of $f.$

If $\Z_{f}(s,z)$ has a meromorphic continuation which is regular at $s=0,$ we define the \emph{superzeta regularized product} associated to $f$ as
$$
D_{f}\left( z \right) = \exp\left( {-\frac{d}{ds}\left. \Z_{f}\left( s,z\right) \right|_{s=0}  } \right).
$$

Hadamard's product formula allows us to write
\beq
f(z) = \Delta_{f}(z) = e^{g(z)} z^r \prod_{k=1}^\infty \left( \left(1-\frac{z}{y_k}
\right)\exp\left[ \frac{z}{y_k} + \frac{z^2}{2{y_k}^2}   \right]    \right),
\eeq
where $g(z)$ is a polynomial of degree 2 or less, $r\geq 0$ is the order of eventual zero of $f$ at $z=0,$ and the other zeros $y_k$ are listed with multiplicity.

The following proposition, originally due to Voros ( \cite{Voros1}, \cite{Voros3}, \cite{VorosKnjiga}) is proven in \cite[Prop. 4.1]{FJS16} (see also \cite{FJC19} for a related, more general result).

\begin{prop} \label{prop: Voros cont.}
Let $f$ be an entire function of order two, and for $k\in\NN,$ let $y_k$ be the sequence of zeros of $f.$ Let $\Delta_{f}(z)$
denote the Hadamard product representation of $f.$ Assume that for $n>2$ we have the following asymptotic expansion:
\begin{equation} \label{defAE}
\log \Delta_{f}(z)= \widetilde{a}_{2}z^{2}(\log z-\frac{3}{2}%
)+b_{2}z^{2}+\widetilde{a}_{1}z\left( \log z-1\right) +b_{1}z+\widetilde{a}%
_{0}\log z+b_{0}+\sum_{k=1}^{n-1}a_{k}z^{\mu _{k}} + h_n(z),
\end{equation}
where $1>\mu _{1}>...>\mu _{n} \rightarrow -\infty $, $\mu_k\neq0,$ and $h_n(z)$ is a sequence of holomorphic functions in the sector $\left\vert \arg z\right\vert <\theta <\pi, \quad (\theta >0)$ such that $h_n^{(j)}(z)=O(|z|^{\mu_n-j})$, as $\left\vert z\right\vert \rightarrow \infty $ in the above sector, for all integers $j \geq 0.$

Then, for all $z\in X_f,$ the superzeta function $\Z_{f}(s,z)$  has a meromorphic continuation to the half-plane $\Re(s)<2$ which is regular at $s=0.$

Furthermore, the superzeta regularized product $D_{f}\left( z\right) $ associated to $f(s)$ is related to $\Delta_{f}(z)$ through the formula
\begin{equation}
\exp\left( {-\frac{d}{ds}\left. \Z_{f}\left( s,z\right) \right|_{s=0}  } \right) = D_{f}(z)=e^{-(b_{2}z^{2}+b_{1}z+b_{0})}\Delta_{f}(z).  \label{D(z)}
\end{equation}
\end{prop}

\subsection{Meromorphic continuation of $Z_+$ and $Z_-$}

Let $X_{\pm} = X_{Z_{\pm}},$ and for $z \in X_{\pm},$ denote by $\mathcal{Z}_{\pm}(s,z) :=\mathcal{Z}_{Z_{\pm}}(s,z)$ the superzeta functions of $Z_{\pm}.$

\begin{prop}\label{prop:superzetas}
For $z \in X_{\pm},$ the superzeta functions  $\mathcal{Z}_{\pm}(s,z)$ have meromorphic continuation to all of $s\in\CC$, regular at $s=0.$ Moreover,
\begin{equation}\label{zeta values at s=0}
\Z_{+}\left( 0,z\right)=-h\frac{\vol(M)}{2\pi}z^2 -\left(\widetilde{a}_1 +\kv\right)z+\kv-\widetilde{a}
_{0} \quad \mathrm{and} \quad \Z_{-}\left( 0,z\right)=- h\frac{\vol(M)}{2\pi}z^2 -\left(\widetilde{a}_1 +\kv\right)z -\widetilde{a}_0+\frac{\kv}{2}
\end{equation}
where $\widetilde{a}_1$ and $\widetilde{a}_0$ are defined by equations \eqref{a1tilde def} and \eqref{widetilde a0}.
\end{prop}
\begin{proof}
We claim that $Z_{+}(z)$ and $Z_{-}(z)$ both are entire, order two functions which satisfy the hypothesis of Proposition~\ref{prop: Voros cont.}.
Indeed, the function $G_1(z)$ is a product of rescaled Barnes double gamma functions, so by using  the asymptotic expansion \eqref{gammaExpan} of the
gamma function, the expansion \eqref{asmBarnes} of the Barnes double gamma function, the bound \eqref{eqSelZetaBound} for logarithm of the Selberg zeta function,
and the asymptotic expansion of the logarithm of the automorphic scattering matrix $\phi(z)=L(z)H(z)$, deduced from \eqref{eqPhiA} and \eqref{asmPhi},  we can obtain an asymptotic expansion of the form
\eqref{defAE} for both $Z_+$ and $Z_-$. We refer to the proof of \cite[Thm. 6.2]{FJS16} where similar computations are worked out in complete detail.
Thus, by Proposition~~\ref{prop: Voros cont.}
both $\mathcal{Z}_{\pm}(s,z)$ have meromorphic continuation to all of $s\in\CC$ which are regular at $s=0.$

To prove the equation \eqref{zeta values at s=0} we need to deduce asymptotic expansion of $\log Z_+(z)$. From the definition~\ref{defCompleteZeta} we have
$$ \log{Z_+(z)} = \log{Z(z)} -\log{G_1(z)} - \kv \log{\Gamma(z-1/2)}$$
hence, combining \eqref{eqExpG1} with
\beq \label{eqGammShift}
\log{\Gamma(z-\tfrac{1}{2})} = z(\log(z)-1)-\log(z) + \frac{1}{2}\log(2\pi) + O(z^{-1})
\eeq
and with \eqref{eqSelZetaBound}, which yields $\log{Z(z)} = O(z^{-1})$ we get, as $\Re(z) \rightarrow \infty$:
\begin{multline}\label{log z+}
\log{Z_+(z)} = -h\frac{\vol(M)}{2\pi}z^{2}(\log z-\frac{3}{2}%
)-\left(\widetilde{a}_1 +\kv\right)z\left( \log z-1\right)\\ -b_{1}z+(\kv-\widetilde{a}
_{0})\log z-b_{0}-\frac{\kv}{2}\log(2\pi)+\sum_{k=1}^{n-1}a_{k}z^{\mu _{k}} + h_n(z),
\end{multline}
where $\widetilde{a}_1$, $b_1$, $\widetilde{a}_0$ and $b_0$ are given by formulas \eqref{a1tilde def}, \eqref{b1 def}, \eqref{widetilde a0} and \eqref{b0} respectively.

From the proof of \cite[Proposition 4.1]{FJS16} (see specifically \cite[Equation 4.9]{FJS16}),  and the asymptotic expansion \eqref{log z+} we see that for $-1<\Re(s)<3$ and $z\in X_+$ we have
$$
\Z_+(s,z)=\frac{-2h\frac{\vol(M)}{2\pi}}{(s-1)(s-2)}z^{2-s}+\frac{\widetilde{a}_1 +\kv}{(s-1)}z^{1-s}+(\kv-\widetilde{a}
_{0})z^{-s}+\frac{1}{\Gamma(s)}f(s,z),
$$
where $f(s,z)$ is a holomorphic function for $\Re(s)>-1/2$. Therefore
$$
\Z_+(0,z)=-h\frac{\vol(M)}{2\pi}z^2 -\left(\widetilde{a}_1 +\kv\right)z+\kv-\widetilde{a}
_{0}
$$
which proves the first formula in \eqref{zeta values at s=0}.

Recall that $Z_{-}(z) = \phi(z) Z_{+}(z).$ Hence from the asymptotic properties of $\log{\phi(z)}$, given in \eqref{eqPhiA}, it follows that
\begin{equation}\label{log z-}
\log{Z_{-}(z)} = \log{Z_{+}(z)} + \frac{\kv}{2}\log{\pi} -\frac{\kv}{2}\log{z} + c_1z + c_2 +o(1)
\,\,\,\,\,\text{\rm as $z \rightarrow \infty$}.
\end{equation}
Here $c_1$ and $c_2$ are from \eqref{eqPhiA}, and we used the asymptotic expansions \eqref{log gamma of z} of $\log\Gamma(z-\tfrac{1}{2})$ and asymptotic
\beq \label{log gamma of z}
\log{\Gamma(z)} = z(\log(z)-1)-\frac{1}{2}\log(z) + \frac{1}{2}\log(2\pi) + O(z^{-1}),
\eeq
of $\log{\Gamma(z)}$. This, together with \eqref{log z+} gives

\begin{multline}\label{asympt logz-}
\log{Z_{-}(z)} = -h\frac{\vol(M)}{2\pi}z^{2}(\log z-\frac{3}{2}
)-\left(\widetilde{a}_1 +\kv\right)z\left( \log z-1\right)\\+(c_1-b_1)z+ \left(\frac{\kv}{2}-\widetilde{a}_{0}\right)\log z-b_{0}-\frac{\kv}{2}\log2+c_2+ \log z\sum_{k=1}^{n-1}a'_{k}z^{\mu _{k}} + g_n(z).
\end{multline}
Reasoning analogously as above, using the asymptotic expansion \eqref{asympt logz-}, we deduce the second formula in \eqref{zeta values at s=0}.

\end{proof}

\section{Main results}\label{sec: main res}
Recall the definition of  $N(Z_+),$ the null set of $Z_+(s),$ which is specified in \S\ref{complete zetas}, and the set  $N(Z_-),$  which is symmetric to $N(Z_+)$ with respect to the mapping $s\mapsto1-s$. One observes that it contains essentially the non-trivial spectral information of the Laplacian:  the eigenvalues and the resonances.

To motivate our key definition we perform a \emph{purely formal computation.} First recall \eqref{Zeta1} the definition of the superzeta functions $\Z_{\pm}\left( s,z\right):=\mathcal{Z}_{Z_{\pm}}(s,z),$
Next, recalling that $s_j(1-s_j) = \lambda_j,$ we have formally
\begin{multline}
 \exp\left( {-\frac{d}{ds}\left.\left( \Z_{+}\left( s,z\right)+\Z_{-}\left( s,z\right) \right) \right|_{s=0}  } \right)= \prod_j \left[ (z - s_j)(z - (1-s_j)) \right]^2(z-\rho_j)(z-(1-\rho_j)) \\
= \prod_j(\lambda_j-z(1-z))^2 \prod_j (\rho_j(1-\rho_j) - z(1-z)),
\end{multline}
where the first product is taken over all discrete eigenvalues of the Laplacian and the second product is taken over all resonances. Therefore, the sum $\Z_{+}\left( s,z\right)+\Z_{-}\left( s,z\right) $ satisfies the same formal identity as the zeta function given by formula (1.3) in \cite{Ef88-91}. This motivates our definition of the regularized determinant of the Laplacian.

\begin{defn} \label{def: square of Lapl} Let $\mathcal{Z}_{\pm}(s,z) :=\mathcal{Z}_{Z_{\pm}}(s,z)$ be the superzeta functions of $Z_{\pm}.$ The square of the regularized determinant of $\lp - z(1-z)I$ is defined to be
\begin{equation} \label{reg det sq defin}
\mathrm{det}^2(\Delta-z(1-z)I)= \exp\left( {-\frac{d}{ds}\left.\left( \Z_{+}\left( s,z\right)+\Z_{-}\left( s,z\right) \right) \right|_{s=0}  } \right)=D_{Z_+}(z)D_{Z_-}(z).
\end{equation}
\end{defn}

Our main result is the following explicit evaluation of $\mathrm{det}^2(\Delta-z(1-z)I)$.

\begin{thm} \label{thm:main} For $z\in  X_{+} \cap X_{-}$, the regularized determinant of the square of $\Delta-z(1-z)I$ is given by the formula
\begin{multline} \label{main result}
\mathrm{det}^2(\Delta-z(1-z)I)= \exp\left((2b_1-c_1) z + 2b_0+\tfrac{\kv}{2}\log(4\pi)-c_2 \right)  \phi(z)\cdot \left( \frac{Z(z)}{G_1(z)(\Gamma(z-1/2))^{\kv}} \right)^2.
\end{multline}
Here $b_1$ is given by \eqref{b1 def}, $b_0$ is defined in \eqref{b0} and $c_1$, $c_2$ are from  \eqref{eqPhiA}.
\end{thm}
\begin{proof}
From the expansion \eqref{log z+} of $\log{Z_+(z)}$, applying Proposition \ref{prop: Voros cont.} we deduce
$$
D_+(z)= \exp(b_1 z+b_0+\tfrac{\kv}{2}\log(2\pi))Z_+(z).
$$
Analogously, the expansion \eqref{asympt logz-} together with Proposition \ref{prop: Voros cont.} yields
$$
D_-(z)= \exp\left((b_1-c_1) z + b_0+\tfrac{\kv}{2}\log2-c_2 \right) Z_-(z).
$$
Combining the expressions for $D_+$ and $D_-$ with definitions of $Z_+$ and $Z_-$ completes the proof.

\end{proof}

Definition \eqref{reg det sq defin} of the square of the regularized determinant is justified by the following functional relation between $D_+(1-z)D_-(1-z)$ and $D_+(z)D_-(z)$

\begin{thm} \label{prop:symmetric funct eq}
For constants $a(\chi)$ defined by \eqref{defAchi}, $b_1$ defined by \eqref{b1 def}, and $c_1$ from  \eqref{eqPhiA}, we have the following symmetric functional equation
$$
\exp(-(2b_1-c_1+2\log a(\chi))z)D_+(z)D_-(z)= \exp(-(2b_1-c_1+2\log a(\chi))(1-z))D_+(1-z)D_-(1-z).
$$
\end{thm}
\begin{proof} First, we will write the function $\Xi(z)$ introduced in the definition 2.1.4. on p. 115 of \cite{Fisher87} (with the weight $k=0$) in our notation:
\begin{multline} \label{def: Xi}
\Xi(z)=\exp\left(h\frac{\vol(M)}{2\pi} z(1-z)\right)\left(\frac{(2\pi)^z \Gamma(z)}{G(z+1)^2}\right)^{h\frac{\vol(M)}{2\pi}}\cdot Z(z)\cdot \\ \cdot \prod_{\{R\}_\Gamma}  d_R^{h(1-1/d_R)z}\left(\Gamma(z)\right)^{-h(1-1/d_R)}\prod_{m=0}^{d_R - 1} \Gamma\left(\frac{z+m}{d_R}\right)^{\alpha(R,m)/d_R} \cdot \\ a(\chi)^{-z}g_1^{-z}(z-1/2)(z-1/2)^{-\tfrac{1}{2} \mathrm{tr}(I_{\mathbf{k}} - \Phi(\tfrac{1}{2}))}\Gamma(z-1/2)^{-\mathbf{k}}\prod_{m=1}^M\left(1+\frac{z-\tfrac{1}{2}}{\sigma_m-\tfrac{1}{2}}\right)\mathcal{P}(z)^{-1}.
\end{multline}
The zeros $1-\sigma_m$ for $m=1\dots T$ of the scattering determinant $\phi$ are defined in \S\ref{ascatmatrix},  $M= \sum\limits_{i=1}^{T}q(\sigma_i)$ is the number of zeros $\sigma_m,$ counted with multiplicities and $\mathcal{P}(z)$ is a function depending on the resonances of $\phi$ which is given in  \cite[Definition 3.2.2, p. 118]{Fisher87}. In our notation (assuming $\mathbf{k}\neq 0,$ i.e. assuming $\chi$ is singular) we have
\begin{multline*}
\mathcal{P}(z) = \prod_{n=1}^N\left(1+\frac{z-\tfrac{1}{2}}{\rho_n-\tfrac{1}{2}}\right)^{-1}
\exp\left(-\frac{1}{2}\left(\frac{z-1/2}{\rho_n-1/2}\right)^2\right)
\cdot \\ \cdot \prod_{\rho}\left(1+\frac{z-\tfrac{1}{2}}{\rho-\tfrac{1}{2}}\right)^{-1}\left(1+\frac{z-\tfrac{1}{2}}{\overline{\rho}-\tfrac{1}{2}}\right)^{-1}
\exp\left(-\frac{1}{2}\left(\frac{z-1/2}{\rho-1/2}\right)^2-\frac{1}{2}\left(\frac{z-1/2}{\overline{\rho}-1/2}\right)^2\right),
\end{multline*}
where the zeros $\rho_i$, $i=1,\ldots,N$ and zeros $\rho$ are are defined in \S\ref{ascatmatrix}.

Using \cite[Formula 3.2.1]{Fisher87} we see that for $\kv\neq0,$ $\mathcal{P}(z)$  satisfies the functional relation
\begin{equation} \label{funct eq P}
\mathcal{P}(1-z)=g_1^{2z-1}\phi^{-1}(\tfrac{1}{2})\prod_{m=1}^M\frac{\sigma_m-z}{\sigma_m+z-1}\mathcal{P}(z) \phi(z),
\end{equation}
where $g_1$ is from \eqref{phiDirich}.

Comparing the definition \eqref{def: Xi} of $\Xi(z)$  with the definition \eqref{defn:g1} of  $G_1(z)$, we arrive at the following relation:
\begin{multline} \label{product RHS}
Z(z)G_1(z)^{-1}\Gamma(z-1/2)^{-\kv} \\ =\Xi(z)\exp(-h\frac{\vol(M)}{2\pi}z(1-z))a(\chi)^{z} \cdot(z-1/2)^{\tfrac{1}{2}\mathrm{tr}(I_\kv + \Phi(\tfrac{1}{2}))}g_1^z \prod_{m=1}^M\left(1+\frac{z-\tfrac{1}{2}}{\sigma_m-\tfrac{1}{2}}\right)^{-1}\mathcal{P}(z),
\end{multline}
Since $\Xi(z)=\Xi(1-z)$ and  $\phi^2(\tfrac{1}{2})=1$, taking the square of the expression \eqref{product RHS}, inserting $1-z$ instead of $z$ and applying the functional relation \eqref{funct eq P}, we deduce the following:
$$
\left( \frac{Z(1-z)}{G_1(1-z)(\Gamma(1/2-z))^{\kv}} \right)^2=a(\chi)^{2(1-2z)} \left( \frac{Z(z)}{G_1(z)(\Gamma(z-1/2))^{\kv}} \right)^2 \phi^2(z),
$$
which is equivalent to equation
$$
\phi(1-z)\left( \frac{Z(1-z)}{G_1(1-z)(\Gamma(1/2-z))^{\kv}} \right)^2=a(\chi)^{2(1-2z)}\phi(z) \left( \frac{Z(z)}{G_1(z)(\Gamma(z-1/2))^{\kv}} \right)^2.
$$
The above equation, together with \eqref{main result} yields that
\begin{multline*}
D_+(1-z)D_-(1-z)= \exp\left((2b_1-c_1)(1- z) + 2b_0+\tfrac{\kv}{2}\log(4\pi)-c_2 \right) a(\chi)^{2(1-2z)}\phi(z)\cdot\\ \cdot \left( \frac{Z(z)}{G_1(z)(\Gamma(z-1/2))^{\kv}} \right)^2,
\end{multline*}
which gives the relation
$$
D_+(1-z)D_-(1-z)= \exp((2b_1-c_1)(1- 2z))a(\chi)^{2(1-2z)}D_+(z)D_-(z).
$$
\end{proof}

\begin{rem}\rm
The expression \eqref{main result} can be compared with the results of \cite{Ef88-91} in the scalar setting ($h=1$) and the surface is torsion-free. Using the relation $$\Gamma_2(z)=(2\pi)^{z/2}\Gamma(z)(G(z+1))^{-1}$$ between the double gamma function $\Gamma_2(z)$ and the Barnes double gamma function $G(z+1),$ we can rewrite \eqref{main result} in the notation of \cite{Ef88-91} (where $h_1$ denotes the degree of singularity):
$$
\mathrm{det}^2(\Delta-z(1-z)I)= (2z-1)^{\tilde{A}}\exp\left(-\left[\tilde{B} (2z-1) + \tilde{C} \right] \right)  \phi(z)\cdot \left(Z(z) Z_\infty(z) \right)^2(\Gamma(z+1/2))^{-2h_1},
$$
for certain, explicitly computable constants $\tilde{A}$, $\tilde{B}$ and $\tilde{C}$.

This expression differs from the corrected formula for the square of the determinant of the Laplacian of \cite{Ef88-91} in constants $\tilde{A}$, $\tilde{B}$ and $\tilde{C}$ and, more importantly, in the factor $\phi(z)$.
(Curiously, the above expression differs from the erroneous formula \cite[Equation 1.7, page 445]{Ef88-91} only in constants $\tilde{A}$, $\tilde{B}$ and $\tilde{C}$.) This is due to a different definition of the regularized determinant of the Laplacian.

The fact that we have introduced the definition of the determinant of the Laplacian which encodes the natural symmetry $z\leftrightarrow 1-z$ shows that the approach to zeta regularization through the superzeta functions carrying the spectral information is better suited in in the presence of the continuous spectrum, than the trace formula approach employed in \cite{Ef88-91}, \cite{Sarnak} and many other more recent papers.
\end{rem}

Our approach immediately yields the following corollary which proves that the scattering determinant can be expressed as the regularized determinant of the superzeta function $\Z_-(s,z) - \Z_+(s,z)$, modulo the factor $\exp(c_1 z+c_2+\tfrac{\kv}{2}\log\pi)$.

\begin{cor} \label{thm:cont}
 For $z\in X_+\cap X_-$
\beq \label{phi expression}
\phi(z)= (\pi)^{\tfrac{\kv}{2}}e^{c_1 z+c_2} \exp\left(-\frac{d}{ds}\left.\left( \Z_-(s,z) - \Z_+(s,z)\right)\right|_{s=0} \right).
\eeq
\end{cor}

\begin{proof}

Utilizing the equation \eqref{log z-} and applying \eqref{D(z)} to both $Z_{-}(s),$ $Z_{+}(s)$ we see that
$$\frac{D_-(z)}{D_+(z)}=\exp\left(-\frac{\kv}{2}\log\pi - c_1z-c_2 \right)\frac{Z_-(z)}{Z_+(z)},$$
hence
$$
\phi(z)=\frac{Z_-(z)}{Z_+(z)}= \exp\left(\frac{\kv}{2}\log\pi + c_1z+c_2 \right)\frac{\exp\left(-\frac{d}{ds}\left.\Z_-(s,z)\right|_{s=0} \right)}{\exp\left(-\frac{d}{ds}\left. \Z_+(s,z)\right|_{s=0} \right)}.
$$
This proves \eqref{phi expression}.

\end{proof}

\appendix

\section{Alternate expressions}

The multiplicities $m_n$ of the trivial zeros of the Selberg zeta function carry important information related to the surface $M$ and the character $\chi$. For this reason, we will deduce a different expression for $m_n$  (see equation \eqref{mn simplified} below)  from which it will be obvious that $m_n$ are non-negative integers. Moreover, we will construct a different order-two entire function $\tilde G_1(s)$, given in terms of the gamma and the Barnes double gamma function (see equation \eqref{G1tilde defn} below) and such that its null set coincides with the set of negative integers $-n$ with multiplicities $m_n$.

\subsection{An alternate expression for the multiplicities $m_n$}
We simplify \eqref{EllipZero} by combining it with \eqref{GaussBon} multiplied by $\frac{h}{2\pi}$ and \eqref{eqDiagEig} to get

\begin{multline} \label{eqMn}
m_n
= h\left( 2g-2 + \csp + \elp \right)(2n+1) -  \sum_{j=1}^h \sum_{\{R\}_{\Gamma}} \frac{1}{d_R}  \left( 2n+1 + \sum_{k=1}^{d_R-1} \left(  \omega(R)^k_j \right) \frac{\sin\left(\frac{k\pi (2n+1)}{d_R}\right)}{\sin\left(\frac{k\pi}{d_R}\right)} \right)
\end{multline}

By \eqref{eqMn} we can focus on the case of unitary characters rather than unitary representations.

\begin{lem} \label{lemFloor1} Let $\omega$ be a unitary character of the finite cyclic group $\left< R \right>,$ where $R$ is elliptic of order $d.$ Further let $\omega(R)=\exp(2 \pi i q/d)$ for some integer $q,$ with $0 \leq q \leq d-1.$ Let $n \in \NN =  \{0,1,2,\dots\},$ then
\begin{equation}\label{sum elliptic}
\frac{1}{d}\left( 2n+1 + \sum_{k=1}^{d-1}\frac{\omega^k(R)}{\sin\left(\frac{k\pi}{d}\right)}\sin\left(\frac{k\pi(2n+1)}{d}\right)\right) =  \left| \{t \in \ZZ ~|~ td \in \{-n+q,\dots n+q\}  \} \right|,
\end{equation}
where $|A|$ denotes the cardinality of the finite set $A$.
\end{lem}
\begin{proof}
We apply the identity
$$
\frac{\sin(k\pi(2n+1)/d)}{\sin(k\pi/d)}= \sum_{j=-n}^{n}\exp(2\pi i k j/d)
$$
and obtain
\begin{multline}
\frac{1}{d}\left( 2n+1 + \sum_{k=1}^{d-1}\frac{\omega^k(R)}{\sin\left(\frac{k\pi}{d}\right)}\sin\left(\frac{k\pi(2n+1)}{d}\right)\right) = \frac{1}{d}(2n+1) + \sum_{k=1}^{d-1} \frac{1}{d} \exp(2 \pi i k q/d)\sum_{j=-n}^{n}\exp(2\pi i k j/d) \\ = \frac{1}{d}(2n+1) +  \frac{1}{d} \sum_{j=-n}^{n} \sum_{k=1}^{d-1} \exp(2 \pi i kq/d) \exp(2\pi i k j/d) = \frac{1}{d} \sum_{j=-n}^{n} \sum_{k=0}^{d-1} \exp(2 \pi i k (q+j)/d)
\end{multline}
The inner sum on the right is equal to $d$ iff $d | (q+j).$ As $j$ runs from $-n,-n+1,\dots,n,$ the number of such $j$ is equal to the number of multiples of $d$ in the integer range $-n+q,\dots,n+q.$ Recalling the factor $\tfrac{1}{d}$ out in front proves the lemma.
\end{proof}

\begin{lem} \label{lemComb1} Let $d,q,n$ be integers with $2 \leq d,$  $0 \leq q \leq d-1,$ and $0 \leq n.$

Then $$ \left| \{t \in \ZZ ~|~ td \in \{-n+q,\dots n+q\}  \} \right|  = \lfloor \frac{n+q}{d} \rfloor + \lfloor \frac{n+d-q}{d} \rfloor $$
Here $\lfloor x\rfloor$ denotes the \emph{floor} function.
\end{lem}
\begin{proof}
For $a \leq b,$ both integers, define $f(a,b,d) = \left| \{t \in \ZZ ~|~ td \in \{a,\dots b\}  \} \right|$

Now, let $a,b$ be integers with $a < 0$ and and $b > 0.$  Then
\begin{multline} \label{eqFloor1}
f(a,b,d) = f(-a,0,d) +  f(0,b,d) - 1 = f(0,-a,d) + f(0,b,d) + 1 = \lfloor \frac{d-a}{d} \rfloor  +  \lfloor \frac{b}{d} \rfloor
\end{multline}

To prove the lemma, we consider the following cases. The first case, $a = -n+q < 0,$ and $b=n+q > 0$ follows from \eqref{eqFloor1}.

In the case when $n=0,$ the lemma is trivial.

Next, consider the case where $0 < -n + q.$
Since $q < d,$  and $n \geq 0,$ it follows that $0<-n+q < d.$ If $n+q < d,$ then it follows that
$$f(-n+q,n+q,d) = 0 =  \lfloor \frac{n+q}{d} \rfloor + \lfloor \frac{d + (n-q)}{d} \rfloor = 0 + 0, $$ and the lemma is verified in this case.

Finally we are left with the case $0 <-n+q < d,$ and $d < n+q.$  We shift the integer interval  $\{-n+q,\dots n+q\}$ to the left by $d,$ and obtain
$$\left| \{t \in \ZZ ~|~ td \in \{-n+q,\dots n+q\}  \} \right| = \left| \{t \in \ZZ ~|~ td \in \{-n+q-d,\dots n+q-d\}  \} \right|. $$
We can apply \eqref{eqFloor1} to the shifted interval and we obtain
\begin{multline}
 f(-n+q-d,n+q-d,d ) = \lfloor \frac{d-(-n+q-d)}{d} \rfloor  +  \lfloor \frac{n+q-d}{d} \rfloor  \\ = \lfloor 1 + \frac{n+d-q}{d} \rfloor  +  \lfloor -1 + \frac{n+q}{d} \rfloor = \lfloor  \frac{n+d-q}{d} \rfloor  +  \lfloor  \frac{n+q}{d} \rfloor,
\end{multline}
where the last equality follows because both $n+q$ and $n+d-q$ are positive.
\end{proof}
One should note that the above Lemma~\ref{lemComb1} is false for arbitrary $n,q,d.$

Combining \eqref{eqMn}, Lemma~\ref{lemFloor1}, and Lemma~\ref{lemComb1} we arrive at the following alternate expression for $m_n$:

\beq \label{mn simplified}
m_n = h\left( 2g-2 + \csp + \elp \right)(2n+1) - \sum_{\{R\}_{\Gamma}} \sum_{j=1}^h \left(  \lfloor \frac{n+q(R)_j}{d_R} \rfloor + \lfloor \frac{n+d_R-q(R)_j}{d_R} \rfloor  \right)
\eeq

\subsection{Double gamma function representation of the trivial zeros of the Selberg-zeta function}

Recall that the Barnes double Gamma function, which is an entire order two function defined by \eqref{BarnesDef} has a zero of multiplicity $n,$ at each point $-n \in \{-1,-2,\dots \}.$ We start with the following lemma.

\begin{lem} \label{lemBarnesRep}
 Let $d,q,n$ be integers with $2 \leq d,$  $0 \leq q \leq d-1,$ and $0 \leq n.$ Set $$g(n,q,d) = \lfloor \frac{n+q}{d} \rfloor + \lfloor \frac{n+d-q}{d} \rfloor. $$ For $s \in \CC,$ define
 $$G_{q,d}(s) = \prod_{m=0}^{d-1} G\left(\frac{s-q+m}{d} + 1 \right) G\left(\frac{s-(d-q)+m}{d} + 1 \right).$$
Then  $G_{q,d}(s)$ is an entire function with the set of zeros being the set of negative integers $\{-n:n = 1,2,\dots \}$ and each zero $s =-n$, $n = 1,2,\dots $ has multiplicity $g(n,q,d)$.
\end{lem}
\begin{proof}
We study $G_{q,d}(s)$ at $s = -n.$ Since $m \in \{0,\dots,d-1\},$  the number $\frac{-n-q+m}{d}$ is a negative integer for exactly one value, say $m = m_1,$ in which case
$$\lfloor \frac{-n -q +m_1}{d} \rfloor  = -\lfloor \frac{n+q}{d} \rfloor .$$

 Therefore, $\prod_{m=0}^{d-1} G\left(\frac{s-q+m}{d}+1\right)$ has a zero of order $\lfloor \frac{n+q}{d} \rfloor$ at $s = -n$ for any positive integer $n$. Similarly, $\prod_{m=0}^{d-1}G\left(\frac{s-(d-q)+m}{d} + 1 \right)$ has a zero of order $\lfloor \frac{n+q-d}{d} \rfloor$ at $s = -n.$

Moreover, $\frac{s-q+m}{d} = -k $ for some $k = 1,2,3,\dots,$ iff $s = -kd + q - m$ is a negative integer. This shows that there are no zeros of $\prod_{m=0}^{d-1} G\left(\frac{s-q+m}{d}+1\right)$ different from negative integers.  A similar conclusion holds for $\prod_{m=0}^{d-1}G\left(\frac{s-(d-q)+m}{d} + 1 \right)$ and the proof is complete.
\end{proof}

Recall that $\{R\}_{\Gamma}$ are the classes of elliptic elements of $\Gamma$ and that there are $\elp$ of them. Further, recall the notation $\csp,h.$ Recall that $\omega(R)_j,$ for $j=1\dots h$ are the eigenvalues (and $d_R\mathrm{th}-$roots of unity) of $\chi(R) $, and that $q(R)_j,$ with $0 \leq q(R)_j \leq d_R-1$ is defined by \eqref{def:q(R)j}.

\begin{defn} \label{defGamE} With the notation above we define
$$G_E(s) =   \prod_{\{R\}_{\Gamma}} \prod_{j=1}^{h} \prod_{m=0}^{d_R - 1} G\left(\frac{s-q(R)_j+m}{d_R} + 1 \right) G\left(\frac{s-\left(d_R-q(R)_j\right)+m}{d_R} + 1 \right).  $$
and
\begin{equation} \label{G1tilde defn}
\tilde G_1(s)= \left( G_E(s) \right)^{-1}\left(\frac{(2\pi)^{-s} (G(s+1))^2}{\Gamma(s)}\right)^{h(2g-2+\csp + \elp)}
\end{equation}
\end{defn}

\begin{lem}
The function $\tilde G_{1}(s)$ is a entire  of order two with zeros at points
$-n \in -\NN$ and corresponding multiplicities  $m_n.$
\end{lem}
\begin{proof}
The function $\left(\frac{(2\pi)^{-s} (G(s+1))^2}{\Gamma(s)}\right)^{h(2g-2+\csp + \elp)}$ possesses zeros at points $s=-n$ with multiplicity $h(2g-2+\csp + \elp)(2n+1)$, hence the statement follows by combining equation \eqref{eqMn} with  Lemmas~\ref{lemFloor1}, \ref{lemComb1} and \ref{lemBarnesRep}.
\end{proof}


\vspace{5mm}

\noindent
Joshua S. Friedman \\
Department of Mathematics and Science \\
\textsc{United States Merchant Marine Academy} \\
300 Steamboat Road \\
Kings Point, NY 11024 \\
U.S.A. \\
e-mail: FriedmanJ@usmma.edu, joshua@math.sunysb.edu, CrownEagle@gmail.com

\vspace{5mm}
\noindent
Jay Jorgenson \\
Department of Mathematics \\
The City College of New York \\
Convent Avenue at 138th Street \\
New York, NY 10031
U.S.A. \\
e-mail: jjorgenson@mindspring.com

\vspace{5mm}

\noindent
Lejla Smajlovi\'c \\
Department of Mathematics \\
University of Sarajevo\\
Zmaja od Bosne 35, 71 000 Sarajevo\\
Bosnia and Herzegovina\\
e-mail: lejlas@pmf.unsa.ba
\end{document}